\documentclass[12pt]{amsart}

\usepackage{amssymb}
\usepackage{bm}
\usepackage{graphicx}
\usepackage{psfrag}
\usepackage{color}
\usepackage{url}

\newcommand{\excise}[1]{}

\newtheorem{thm}{Theorem}[section]
\newtheorem{lemma}[thm]{Lemma}

\newtheorem{cor}[thm]{Corollary}
\newtheorem{prop}[thm]{Proposition}

\newtheorem{prob}[thm]{Problem}

\theoremstyle{definition}
\newtheorem{example}[thm]{Example}
\newtheorem{remark}[thm]{Remark}
\newtheorem{defn}[thm]{Definition}

\numberwithin{equation}{section}

\newcommand{\ring}[1]{\ensuremath{\mathbb{#1}}}

\renewcommand\>{\rangle}
\newcommand\<{\langle}

\newcommand\NN{\ring{N}}

\newcommand\kk{\Bbbk}
\newcommand\mm{{\mathfrak m}}

\newcommand\cB{{\mathcal B}}

\newcommand\app{\mathord\approx}

\newcommand\iso{\cong}
\newcommand\nil{\infty}

\newcommand\til{\mathord\sim}
\newcommand\too{\longrightarrow}

\renewcommand\implies{\Rightarrow}

\def\ol#1{{\overline {#1}}}

\DeclareMathOperator\Hom{Hom} 
\DeclareMathOperator\image{Im} 
\DeclareMathOperator\ann{ann} 
\def\QMod{$Q$-$\mathrm{Mod}$}

\begin{document}

\mbox{}
\title{Mesoprimary decomposition of binomial submodules\qquad}
\author{Christopher O'Neill}
\address{Mathematics Department\\Texas A\&M University\\College Station, TX 77840}
\email{www.math.tamu.edu/\~{\hspace{-.3ex}}coneill}


\begin{abstract}
\hspace{-2.05032pt}
Recent results of Kahle and Miller give a method of constructing primary decompositions of binomial ideals by first constructing ``mesoprimary decompositions'' determined by their underlying monoid congruences.  These mesoprimary decompositions are highly combinatorial in nature, and are designed to parallel standard primary decomposition over Noetherean rings.  In this paper, we generalize mesoprimary decomposition from binomial ideals to ``binomial submodules'' of certain graded modules over the corresponding monoid algebra, analogous to the way primary decomposition of ideals over a Noetherean ring $R$ generalizes to $R$-modules.  The result is a combinatorial method of constructing primary decompositions that, when restricting to the special case of binomial ideals, coincides with the method introduced by Kahle and Miller.  
\end{abstract}
\maketitle


\vspace{-0.1in}

\section{Introduction}\label{s:intro}

Fix a field $\kk$ and a commutative monoid $Q$.  A \emph{binomial ideal} in the monoid algebra $\kk[Q]$ is an ideal $I$ whose generators have at most two terms.  The quotient $\kk[Q]/I$ by a binomial ideal identifies, up to scalar multiple, any monomials appearing in the same binomial in $I$.  This induces a \emph{congruence} $\til_I$ on the monoid $Q$ (an equivalence relation perserving additivity), and the quotient module $\kk[Q]/I$ is naturally graded with a decomposition into 1-dimensional $\kk$-vector spaces, at most one per $\til_I$-class.  In \cite{kmmeso}, Kahle and Miller introduce \emph{mesoprimary decompositions}, which are combinatorial approximations of primary decompositions of $I$ constructed from the congruence~$\til_I$.  

Mesoprimary decomposition of binomial ideals is motivated by combinatorially constructed primary decompositions of monomial ideals.  Any monomial ideal $I$ in the monoid algebra $\kk[Q]$ is uniquely determined by the monomials it contains.  Taking the quotient $\kk[Q]/I$ amounts to setting these monomials to 0, and the monomials that lie outside of $I$ naturally grade the quotient $\kk[Q]/I$ with a decomposition into 1-dimensional $\kk$-vector spaces.  

An irreducible decomposition for a monomial ideal $I$ whose components are themselves monomial ideals can be constructed by locating \emph{witness monomials} ${\bf x}^w$ whose annihilator modulo $I$ is prime, and then constructing for each the primary monomial ideal that contains of all monomials not lying below ${\bf x}^w$.  The intersection of these ideals (one per witness monomial) equals $I$, and the witnesses are readily identified from the grading on $\kk[Q]/I$.  See~\cite[Chapter~5]{cca} for a full treatment of monomial irreducible decomposition.  

Combinatorially constructed irreducible decompositions of monomial ideals have also been shown to live within a larger categorical setting.  Much in the way primary decomposition of ideals over a Noetherean ring $R$ generalizes to $R$-modules, combinatorial methods for constructing primary decompositions of monomial ideals can be generalized to certain modules whose gradings resemble the fine gradings of monomial quotients.  See \cite{Mil02} for an overview of these constructions and \cite{hmcech,Mil00} for consequences.  

Kahle and Miller use congruences to extend the above construction from monomial ideals to binomial ideals \cite{kmmeso}.  Given a binomial ideal $I$, they pinpoint a collection of monomials in $\kk[Q]/I$ that behave like witnesses.  For each witness ${\bf x}^w$, they construct the \emph{coprincipal component} at ${\bf x}^w$, a binomial ideal containing $I$ whose quotient has ${\bf x}^w$ as the unique greatest nonzero monomial.  The resulting collection of ideals, one for each witness, decomposes $I$, and each component admits a canonical primary decomposition.  In this way, mesoprimary decompositions act as a bridge to primary components of a binomial ideal from the combinatorics of its induced congruence.  

Mesoprimary decompositions are constructed in two settings: first for monoid congruences, and then for binomial ideals; both are designed to parallel standard primary decomposition in a Noetherian ring $R$.  This motivated Kahle and Miller to pose Problems~\ref{pr:kmqmod} and~\ref{pr:kmbinmod} below, which appeared as~\cite[Problem~17.11]{kmmeso} and~\cite[Problem~17.13]{kmmeso}, respectively.  These problems, in turn, serve to motivate the results in this paper.  

\begin{prob}\label{pr:kmqmod}
Generalize mesoprimary decomposition of monoid congruences to congruences on monoid modules.  
\end{prob}

\begin{prob}\label{pr:kmbinmod}
Develop a notion of binomial module over a commutative monoid algebra, and generalize mesoprimary decomposition of binomial ideals to this setting.  
\end{prob}

One of the largest tasks in generalizing the results of \cite{kmmeso} to monoid modules is to separate which constructions should happen in the monoid and which should happen in the module, since these coincide for monoid congruences.  See Remarks~\ref{r:cancellative} and~\ref{r:associatedcong} for specific instances of this distinction.  

In the first part of this paper (Sections~\ref{s:qmod}-\ref{s:qmodmesodecomp}), we introduce the category \QMod~of modules over a monoid $Q$ (Definition~\ref{d:qmodule}) and generalize nearly every result from \cite{kmmeso} on monoid congruences to congruences on monoid modules.  We define primary and mesoprimary monoid module congruences (Definition~\ref{d:qmodprimary}) and give equivalent conditions for these congruences in terms of associated objects and witnesses (Theorems~\ref{t:qmodprimary} and~\ref{t:qmodmesoprimary}).  We then construct a mesoprimary decomposition, with one component per key witnesses, for any monoid module congruence (Theorem~\ref{t:qmodmesodecomp}).  The resulting theory completely answers Problem~\ref{pr:kmqmod}.  

The second part of this paper (Sections~\ref{s:tightlygraded}-\ref{s:binmesodecomp}) answers Problem~\ref{pr:kmbinmod}.  We introduce the category $\cB_Q$ (Definition~\ref{d:bmod}), whose objects are tightly graded modules (Definition~\ref{d:graded}) over the monoid algebra $\kk[Q]$ graded by monoid modules in \QMod.  It is in this setting that we define binomial submodules (Definition~\ref{d:binomialsubmod}).  We define mesoprimary submodules (Definition~\ref{d:binmesoprimary}) and associated mesoprime ideals (Definition~\ref{d:assocmesoprime}), developing a theory of mesoprimary decomposition (Theorem~\ref{t:binmesodecomp}) that parallels results in \cite{kmmeso}.  In~particular, the binomial submodules of the free module $\kk[Q]$ are precisely the binomial ideals; see Example~\ref{e:mesoprimaryideal}.  

To conclude the paper, we demonstrate in Section~\ref{s:binprimarydecomp} how a binomial primary decomposition may be recovered from a mesoprimary decomposition when the underlying field is algebraically closed.  


\subsection*{Notation}

Throughout this paper, assume $Q$ is a Noetherian commutative monoid and $\kk$ is an arbitrary field, and let $\kk[Q]$ denote the monoid algebra over $Q$ with coefficients in $\kk$.  Unless otherwise stated, all $\kk[Q]$-modules are assumed to be finitely generated.  

\section{The category of monoid modules}%
\label{s:qmod}

In this section, we define the category \QMod~of modules over a commutative monoid $Q$ and extend some of the fundamental concepts and results from monoid ideals and congruences to the objects of this category.  The content of this section (as well as Section~\ref{s:tightlygraded}) is motivated by Example~\ref{e:modqmodule}.  First, we record some preliminary definitions (see \cite{grilletActions} for more detail).  

\begin{defn}\label{d:qmodule}
Fix a commutative monoid $Q$.  
\begin{enumerate}
\item 
A \emph{$Q$-module} $(T,\cdot)$ is a set $T$ together with a left action by $Q$ that satisfies $0 \cdot t = t$ and $(q + q')\cdot t = q\cdot(q' \cdot t)$ 
for all $t \in T$, $q, q' \in Q$.  A subset $T' \subset T$ is a \emph{submodule} of $T$ if it is closed under the $Q$-action, that is, $Q \cdot T' \subset T'$.  The submodule of $T$ \emph{generated by elements} $t_1, \ldots, t_r \in T$ is $\<t_1, \ldots, t_r\> = \bigcup_{i = 1}^r Q \cdot t_i.$

\item 
A map $\psi:T \to U$ between $Q$-modules $T$ and $U$ is a \emph{$Q$-module homomorphism} if $\psi(q \cdot t) = q \cdot \psi(t)$ for all $t \in T, q \in Q$.  The set of $Q$-module homomorphisms from $T$ to $U$ is denoted by $\Hom_Q(T,U)$, and is naturally a $Q$-module with action $q \cdot \psi$ given by $(q \cdot \psi)(t) = \psi(q \cdot t)$.  

\item 
The \emph{category of $Q$-modules}, denoted \QMod, is the category whose objects are $Q$-modules and whose morphisms are $Q$-module homomorphisms.  

\end{enumerate}
\end{defn}

Direct sums, direct products, and tensor products exist in the category \QMod.  We now state their constructions explicitly.  

\begin{defn}\label{d:qsumprod}
Fix two $Q$-modules $T$ and $U$.  
\begin{enumerate}
\item 
The \emph{direct sum} $T \oplus U$ is the disjoint union $T \coprod U$ as sets, with the natural $Q$-action on each component.  

\item 
The \emph{direct product} $T \times U$ is the cartesian product of $T$ and $U$ as a set, with componentwise $Q$-action.  

\item 
The \emph{tensor product} $T \otimes_Q U$ is the collection of formal elements $t \otimes u$ for $t \in T$ and $u \in U$ modulo the equivalence relation generated by 
$$t \otimes (q \cdot u) \sim (q \cdot t) \otimes u \text{ for } t \in T \text{ and } u \in U$$
The action of $Q$ is given by $q \cdot(t \otimes u) = (q \cdot t) \otimes u$ for $q \in Q$, $t \in T$ and $u \in U$.  

\end{enumerate}
\end{defn}

\begin{defn}\label{d:congruence}
Fix a $Q$-module $T$.  A \emph{congruence on $T$} is an equivalence relation $\til$ on $T$ that satisfies $t \sim t' \Rightarrow q \cdot t \sim q \cdot t'$ for all $q \in Q$ and $t, t' \in T$.  The \emph{quotient module} $T/\til$ is the set of equivalence classes of $T$ under $\til$.  The congruence condition on $\til$ ensures that $T/\til$ has a well defined action by $Q$.  
\end{defn}

\begin{defn}\label{d:ideal}
A subset $T \subset Q$ is an \emph{ideal} if it is a $Q$-submodule of $Q$, that is, $Q + T \subset T$.  An ideal $P \subset Q$ is \emph{prime} if its complement in $Q$ is a submonoid of~$Q$.  
\end{defn}

\begin{defn}\label{d:localization}
Fix a $Q$-module $T$, a prime ideal $P \subset Q$, and set $F = Q \setminus P$.  The \emph{localization of $T$ at $P$}, denoted $T_P$, is the set $T \times F$ modulo the equivalence relation $\til$ that sets $(t,f) \sim (t',f')$ whenever $w \cdot f' \cdot t = w \cdot f \cdot t'$ for some $w \in Q$.  The localization $Q_P$ is naturally a monoid, and $T_P$ is naturally a $Q_P$-module.  Write $t-f$ to denote the element $(t,f) \in T \times F$.  
\end{defn}

\begin{remark}\label{r:inducedconglocal}
Any congruence $\til$ on a $Q$-module $T$ induces a congruence on $T_P$.  
\end{remark}

\begin{defn}\label{d:greensorder}
Fix a $Q$-module $T$.  \emph{Green's preorder} on $T$ sets $t \preceq t'$ whenever $\<t\> \supset \<t'\>$.  \emph{Green's relation} on $T$ sets $t \sim t'$ whenever $\<t\> = \<t'\>$.  
\end{defn}

Green's preorder on a monoid orders its elements by divisibility, and this notion extends to $Q$-modules.  

\begin{lemma}\label{l:greensorder}
Green's relation $\til$ on a $Q$-module $T$ is a congruence on $T$, and the quotient $T/\til$ is partially ordered by divisibility.  
\end{lemma}

\begin{proof}
For $t, t' \in T$ and $q \in Q$, we can see $\<t\> = \<t'\>$ implies $\<q \cdot t\> = \<q \cdot t'\>$.  Each element of the quotient $T/\til$ generates a distinct submodule, so the divisibility preorder is antisymmetric, and thus a partial order.  
\end{proof}

We now generalize the notion of a nil element of a monoid.  

\begin{defn}\label{d:nil}
An element $\nil \in T$ in a $Q$-module $T$ is called a \emph{nil} if it is absorbing, that is, $Q \cdot \nil = \{\nil\}$.  The \emph{basin} of a nil $\nil \in T$ is the set 
$$B(\nil) = \{t \in T : qt = \nil \text{ for some } q \in Q\}$$
of elements of $T$ that can be sent to $\nil$ under the action of $Q$.  The \emph{nil set} of $T$, denoted $N(T)$, is the collection of all nil elements in $T$.  
\end{defn}

\begin{defn}\label{d:properly}
Fix a subset $U \subset T$ of a $Q$-module $T$.  A \emph{$Q$-orbit} of $U$ is a connected component of the undirected graph whose vertices are elements of $U$ and whose edges connect two vertices $s, t \in U$ whenever $q \cdot s = t$ for some $q \in Q$.  $T$ is \emph{connected} if it has at most one $Q$-orbit, and $T$ is \emph{properly connected} if $T \setminus N(T)$ has at most one $Q$-orbit.  
\end{defn}

\begin{example}\label{e:modproperly}
Let $T$ and $U$ be connected $Q$-modules with nils $\infty_T$ and $\infty_U$, respectively.  If $T \setminus \{\infty_T\}$ and $U \setminus \{\infty_U\}$ are both nonempty, the module $(T \coprod U)/\<\nil_T \sim \nil_U\>$ is connected and has a single nil, but it is not properly connected, since removing the nil produces two distinct $Q$-orbits.  
\end{example}

\begin{remark}\label{r:nils}
Unlike a monoid, a $Q$-module may have more than one nil element.  However, by Lemma~\ref{l:nilorbit}, each $Q$-orbit can have at most one nil element.  
\end{remark}

\begin{lemma}\label{l:nilorbit}
The basin of a nil element $\nil \in T$ in a $Q$-module $T$ is the $Q$-orbit of $T$ containing $\nil$.  
\end{lemma}

\begin{proof}
The basin of $\nil$ is clearly contained in its $Q$-orbit, and whenever $qt = s$ for $q \in Q$ and $s, t \in T$,  we have $t \in B(\nil)$ if and only if $s \in B(\nil)$.  
\end{proof}

\begin{example}\label{e:modqmodule}
Let $Q = \NN^2$, $I = \<x^2,y^2\> \subset \kk[Q]$, $R = \kk[Q]/I$, and 
$$M = (R \oplus R)/\<xye_1 - xye_2\>,$$
where $e_1$ and $e_2$ generate the free $\kk[Q]$-module $R \oplus R$.  $R$ is graded by the quotient monoid $Q/\til_I$, and $M$ is graded by two disjoint copies of $Q/\til_I$ with both copies of $xy$ and the nil elements identified.  Unlike the monoid that grades $R$, this grading does not have a natural monoid structure.  It does, however, have a natural action by $Q$, correponding to the action on $M$ by monomials in $\kk[Q]$.   See Figure~\ref{fig:monoidmodule} for an illustration.  
\end{example}

\begin{figure}[tbp]
\begin{center}
\includegraphics[width=2.5in]{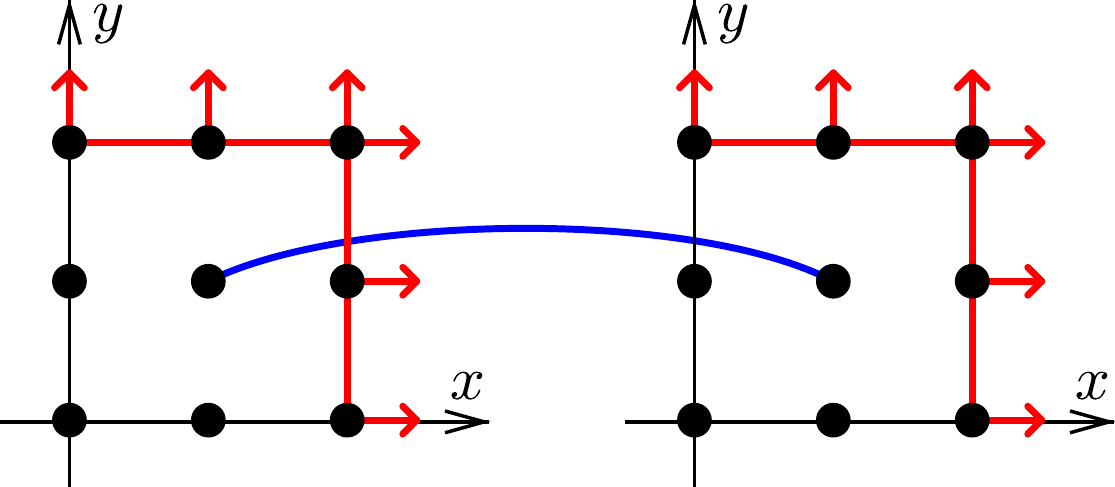}
\end{center}
\caption[A monoid module over $\NN^2$]{The monoid module that grades $M$ in Example~\ref{e:modqmodule}.}
\label{fig:monoidmodule}
\end{figure}

There is also a notion of decomposition of $Q$-modules into indecomposables.  

\begin{lemma}\label{l:orbitdecomp}
Every $Q$-module $T$ has a unique decomposition $T = \bigoplus_i T_i$ as a direct sum of connected modules.  
\end{lemma}

\begin{proof}
Any $Q$-module is the disjoint union of its $Q$-orbits.  
\end{proof}

\begin{remark}\label{r:kernel}
Kernels, in the categorical sense, do not exist in the category \QMod.  However, there is still a notion of kernel of a $Q$-module homomorphism as a congruence; see Definition~\ref{d:kernel}.  This definition is justified by Theorem~\ref{t:isothm}, a $Q$-module analogue of the first isomorphism theorem for groups.  
\end{remark}

\begin{defn}\label{d:kernel}
Fix a homomorphism $\phi:T \to U$.  The \emph{kernel of $\phi$}, denoted $\ker(\phi)$, is the congruence $\til$ on $T$ that sets $t \sim t'$ whenever $\phi(t) = \phi(t')$ for $t, t' \in T$.  
\end{defn}

\begin{thm}\label{t:isothm}
If $\phi:T \to U$ is a $Q$-module homomorphism, then $T/\ker(\phi) \iso \image(\phi)$.  
\end{thm}

\begin{proof}
The homomorphism $\phi$ is surjective onto its image, and the quotient of $T$ by $\ker(\phi)$ identifies elements with the same image under $\phi$.  This ensures that the map $T/\ker(\phi) \too \image(\phi)$ is both injective and surjective.  
\end{proof}

\begin{cor}\label{c:modstruct}
Any finitely generated $Q$-module $T$ is isomorphic to a quotient of a direct sum of finitely many copies of $Q$.  
\end{cor}

\begin{proof}
Fix a finitely generated $Q$-module $T = \<t_1, \ldots, t_r\>$.  Let $\phi:\bigoplus_{i=1}^r Q \too T$, where the map on the $i$-th summand is given by $Q \to \<t_i\>$.  This map is surjective, so Theorem~\ref{t:isothm} implies $T \iso (\bigoplus_{i=1}^r Q)/\ker(\phi)$.  
\end{proof}

\section{Primary and mesoprimary monoid modules}%
\label{s:qmodmesoprimary}

Mesoprimary decomposition of monoid congruences models primary decomposition of ideals in a Noetherian ring $R$, with mesoprimary congruences playing the role of primary ideals and prime congruences playing the role of prime ideals.  In this section, we generalize the notion of mesoprimary monoid congruences to congruences on monoid modules (Definition~\ref{d:qmodprimary}), analogous to the way primary decomposition of ideals in $R$ generalizes to finitely generated $R$-modules.  The main result is Theorem~\ref{t:qmodmesoprimary}, which generalizes \cite[Theorem~6.1]{kmmeso} and characterizes mesoprimary monoid module congruences in terms of their associated prime congruences (Definition~\ref{d:associatedcong}).

\begin{defn}\label{d:cancellative}
Fix a $Q$-module $T$.  For each $q \in Q$, let $\phi_q$ denote the map $T \stackrel{\cdot q}\too T$ given by action by $q$.  
\begin{itemize}
\item 
An element $q \in Q$ acts \emph{cancellatively on $T$} if $\phi_q$ is injective.  

\item 
An element $q \in Q$ acts \emph{nilpotently on $T$} if for each $t \in T$, $(nq) \cdot t \in N(T)$ for some nonnegative integer $n$.  

\item 
An element $t \in T$ is \emph{partly cancellative} if whenever $a \cdot t = b \cdot t \notin N(T)$ for $a, b \in Q$ that act cancellatively on $T$, the morphisms $\phi_a$ and $\phi_b$ coincide.  

\end{itemize}
\end{defn}

\begin{remark}\label{r:cancellative}
Each term in Definition~\ref{d:cancellative} is also defined in \cite[Definition~2.9]{kmmeso} for monoid elements.  However, we are forced to make a distinction between monoid elements and monoid module elements (these objects coincide in the setting of \cite{kmmeso}).  In~particular, ``cancellative'' and ``nilpotent'' (Definition~\ref{d:cancellative}) are properties of monoid elements, whereas ``partly cancellative'' is a property of monoid module elements.  Roughly speaking, ``cancellative'' and ``nilpotent'' describe how a particular $q \in Q$ acts on different module elements, whereas ``partly cancellative'' describes how different monoid elements act on a particular $t \in T$.  
\end{remark}

\begin{defn}\label{d:qmodprimary}
A $Q$-module $T$ is 
\begin{itemize}
\item 
\emph{primary} if each $q \in Q$ is either cancellative or nilpotent on $T$.  

\item 
\emph{mesoprimary} if it is primary and each $t \in T$ is partly cancellative.  

\end{itemize}
A congruence $\til$ on $T$ is \emph{primary} (respectively, \emph{mesoprimary}) if $T/\til$ is a primary (respectively, mesoprimary) $Q$-module.  
\end{defn}

\begin{lemma}\label{l:primarymonoidcase}
Fix a congruence $\til$ on $Q$.  The $Q$-module $T = Q/\til$ is (meso)primary in the sense of Definition~\ref{d:qmodprimary} if and only if $\til$ is a (meso)primary monoid congruence in the sense of \cite{kmmeso}.  
\end{lemma}

\begin{proof}
For $q \in Q$ let $\ol q$ denote the image of $q$ modulo $\til$.  An element $q \in Q$ acts cancellatively on $T$ if and only if its image modulo $\til$ is cancellative, and $q$ acts nilpotently on $T$ if and only if it has nilpotent image modulo $\til$.  This proves that $T$ is a primary $Q$-module if and only if $\til$ is primary as a monoid congruence.  Lastly, assuming $\til$ is $P$-primary, notice that for $a, b \notin P$, $\phi_a = \phi_b$ if and only if $\ol a = \ol b \in T$, so each $\ol q \in T$ is partly cancellative as a monoid element if and only if it is partly cancellative as an element of a $Q$-module.  This completes the proof.  
\end{proof}

We now generalize witnesses and key witnesses from \cite{kmmeso} to the setting of monoid module congruences.  Definition~\ref{d:qmodwitness}, while complex, very closely resembles \cite[Definition~4.7]{kmmeso}; see the original text for several motivating examples.  Key witnesses are used to construct the mesoprimary components (Definition~\ref{d:qmodcoprincomp}) used to decompose monoid module congruences (Theorem~\ref{t:qmodmesodecomp}).  

\begin{defn}\label{d:qmodwitness}
Let $T$ be a $Q$-module, $P \subset Q$ a prime ideal, and $\til$ a congruence on $T$.  For $t \in T$, let $\ol t$ denote the image of $t$ in $\ol T_P$, and for $p \in P$, let $\phi_p:\ol T_P \to \ol T_P$ denote the morphism given by the action of $p$.  

\begin{enumerate}

\item 
An element $w \in T$ is \emph{exclusively maximal} in a set $A \subset \ol T_P$ if $\ol w$ is the unique maximal element of $A$ under Green's preorder.  

\item 
An element $w \in T$ with non-nil image in $\ol T_P$ is a \emph{$\til$-witness} for $P$ if for each generator $p \in P$, the class of $\ol w$ is non-singleton under $\ker(\phi_p)$ and $\ol w$ is not exclusively maximal in that class.  

\item 
An element $w' \in T$ is an \emph{aide} for a $\til$-witness $w$ for $P$ and a generator $p \in P$ if $w$ and $w'$ have distinct images in $\ol T_P$ but are not distinct under $\ker(\phi_p)$.  

\item 
An element $w$  with non-nil image in $\ol T_P$ is a \emph{key $\til$-witness for $P$} if $\ol w$ is non-singleton under $\bigcap_{p \in P} \ker(\phi_p)$ and $\ol w$ is not exclusively maximal in this non-singleton class.  

\item 
The prime $P$ is \emph{associated to $T$} if $T$ has a witness for $P$, or if $P = \emptyset$ and $T$ has a $Q$-orbit with no nil.  

\end{enumerate}
\end{defn}

\begin{remark}\label{r:qmodprimary}
Prior to \cite{kmmeso}, primary decomposition of monoid congruences was developed by Grillet \cite{grilletPrimary}, but these decomposition are too course to effectively recover primary components at the level of binomial ideals \cite[Example~2.22]{kmmeso}.  Nevertheless, in an effort to create a more complete picture, we also generalize primary congruences to our current setting of monoid module congruences.  
\end{remark}

\begin{thm}\label{t:qmodprimary}
A finitely generated $Q$-module $T$ is primary if and only if it has exactly one associated prime ideal.  
\end{thm}

\begin{proof}
Suppose $T$ is primary.  The set of elements with nilpotent action on $T$ is a prime ideal $P \subset Q$.  Since $P$ is finitely generated, some non-nil element $w \in T$ satisfies $P \cdot w \subset N(T)$.  This means $w$ is a witness for $P$, so $P$ is associated to $T$.  Since $Q \setminus P$ acts cancellatively on $T$, any prime associated to $T$ is contained in $P$.  Moreover, localizing $T$ at any prime $P'$ contained in $P$ identifies any element $w \in T$ with the nil in its orbit, since some $p \in P \setminus P'$ gives $p \cdot w \in N(T)$.  Thus, any associated prime must also contain $P$, which implies $P$ is the only associated prime.  

Now suppose $T$ has only one associated prime $P \subset Q$.  If $P = \emptyset$, then every element of $Q$ acts cancellatively on $T$.  Now suppose $P$ is nonempty, and fix $t \in T$.  The submodule $\<t\>$ is isomorphic to $Q$ modulo some congruence.  Since each witness in $\<t\>$ is a witness for $P$, $\<t\>$ is $P$-primary by \cite[Corolary~4.21]{kmmeso}.  This means each $p \in P$ acts nilpotently on $\<t\>$ and each $f \in Q \setminus P$ acts cancellatively on $\<t\>$.  Since $t$ is arbitrary, each $p \in P$ acts nilpotently on $T$ and each $f \in Q \setminus P$ acts cancellatively on $T$, meaning $T$ is $P$-primary.  
\end{proof}

Lemma~\ref{l:greensprimary} generalizes \cite[Lemma~2.19]{kmmeso} and is central to several proofs, including Theorems~\ref{t:qmodmesoprimary} and~\ref{t:qmodmesodecomp}.  

\begin{lemma}\label{l:greensprimary}
Fix a connected, $P$-primary $Q$-module $T$, and set $F = Q \setminus P$.  Let $T/F$ denote the quotient of $T$ by the congruence
$$t \sim t' \text{ whenever } f \cdot t = g \cdot t' \text{ for } f, g \in F$$
Then Green's preorder on $T/F$ is a partial order, and $T/F$ is finite.  
\end{lemma}

\begin{proof}
Since $T$ is $P$-primary, the morphisms $T \stackrel{\cdot f}\too T$ are injective for all $f \in F$, so $\til$ is a well-defined congruence.  If $\<t\> = \<t'\>$, then $f \cdot t = t'$ and $g \cdot t' = t$ for some $f, g \in Q$.  This means $f \cdot g \cdot t = t$, so $f$ and $g$ are not nilpotent and lie in $F$, meaning $t$ and $t'$ are identified in $T/F$.  This proves Green's preorder is antisymmetric.  

Now, the remaining statement is trivial if $P = \emptyset$, so suppose $P$ is nonempty.  $T$ must have a nil $\infty$ since $Q$ contains elements with nilpotent action on $T$.  The image of $\infty$ in $T/F$ remains nil as well.  Thus, since $Q$ and $T$ are both finitely generated, $T/F$ must be finite.  
\end{proof}

\begin{defn}\label{d:associatedcong}
Fix a $Q$-module $T$, a monoid prime $P \subset Q$, and a non-nil $w \in T$.  
\begin{enumerate}
\item Let $G_P \subset Q_P$ denote the unit group of $Q_P$, and let $K_q^P \subset G_P$ denote the stabilizer of $\ol w \in T_P$ under the action of $G_P$.  
\item Let $\app$ denote the congruence on $Q_P$ that sets $a \approx b$ whenever
\begin{enumerate}
\item $a$ and $b$ lie in $P_P$, or 
\item $a$ and $b$ lie in $G_P$ and $a - b \in K_q^P$.  
\end{enumerate}
\item The \emph{$P$-prime congruence of $T$ at $w$} is given by $\ker(Q \to Q_P/\app)$.  
\item The $P$-prime congruence at $w$ is \emph{associated to $T$} if $w$ is a key witness for $T$.  
\end{enumerate}
\end{defn}

\begin{remark}\label{r:associatedcong}
In Definition~\ref{d:associatedcong}, we are forced to make another distinction between $T$ and $Q$: should an associated prime congruence of $T$ be a congruence on $T$ or on $Q$?  The condition for a monoid congruence $\til$ to be $P$-mesoprimary can be characterized in terms of the congruence on $Q \setminus P$ induced by its action on $Q/\til$ \cite[Corollary~6.7]{kmmeso}.  The partly cancellative condition is what ensures that each $t \in T$ induces the same congruence, which in our setting is a condition on elements of $T$.  
\end{remark}

Next, we characterize mesoprimary $Q$-modules in terms of their associated prime congruences, generalizing \cite[Theorem~6.7]{kmmeso} and \cite[Corollary~6.7]{kmmeso}.  

\begin{thm}\label{t:qmodmesoprimary}
For a $Q$-module $T$, the following are equivalent.  
\begin{enumerate}
\item[(1)] 
$T$ is mesoprimary.  

\item[(2)] 
$T$ has exactly one associated prime congruence.  

\item[(3)] 
$T$ is $P$-primary, and for $F = Q \setminus P$, 
$$\ker(F \to \<t\>) = \ker(F \to \<t'\>)$$
for each non-nil $t, t' \in T$.  
\end{enumerate}
\end{thm}

\begin{proof}
From any of these conditions, we conclude that $T$ is primary, say with associated prime $P$.  Notice that $\ker(F \to \<t\>)$ is the prime congruence at $t$ restricted to $F$.  If these congruences coincide for all $t \in T$, then in particular they coincide for all witnesses, so $T$ has exactly one associated prime congruence.  This proves (3) $\implies$ (2).  

Now suppose $T$ is mesoprimary, and fix $t, t' \notin N(T)$.  Then since $t$ and $t'$ are both partly cancellative, $a \cdot t = b \cdot t$ if and only if $a \cdot t' = b \cdot t'$ for $a, b \notin P$.  This means the kernels $\ker(F \to \<t\>)$ and $\ker(F \to \<t'\>)$ coincide.  This proves (1) $\implies$ (3).  

Lastly, suppose $T$ has exactly one associated prime congruence, and fix $t \in N(T)$.  Fix $a, b \notin P$ and let $\phi_a,\phi_b:T \to T$ denote the actions of $a$ and $b$ on $T$, respectively.  By Theorem~\ref{t:isothm}, $\<t\> \iso Q/\til$ for some congruence $\til$.  Since $T$ has only one associated prime congruence, so does $\til$, so by \cite[Theorem~6.1]{kmmeso}, $\til$ is mesoprimary.  This means $a \cdot t = b \cdot t$ if and only $a \cdot w = b \cdot w$ for any witness $w \in \<t\>$.  Since $T$ has only one associated prime congruence, these actions also coincide for all witnesses in $T$, meaning $\phi_a = \phi_b$.  This proves (2) $\implies$ (1), thus completing the proof.  
\end{proof}

We conclude this section with Theorem~\ref{t:finitekey}, which ensures that the mesoprimary decomposition constructed in Theorem~\ref{t:qmodmesodecomp} has finitely many components.  

\begin{thm}\label{t:finitekey}
Any finitely generated $Q$-module $T$ has only finitely many Green's classes of key witnesses.  
\end{thm}

\begin{proof}
Fix a generating set $g_1, \ldots, g_k$ for $T$.  For each $g_i$, consider the map $\phi_i:Q \to \<g_i\>$ and let $\til_i = \ker\phi_i$.  The induced isomorphism $Q/\til_i \to \<g_i\>$ gives a bijection between key $T$-witness and key $\til_i$-witnesses, and by \cite[Theorem~5.6]{kmmeso}, each congruence $\til_i$ has only finitely many Green's classes of key witnesses.  Since $g_1, \ldots, g_k$ generate $T$, this bounds the number of Green's classes of key $T$-witnesses.  
\end{proof}

\section{Mesoprimary decomposition of monoid modules}%
\label{s:qmodmesodecomp}

In this section, we construct a mesoprimary decomposition for any monoid module congruence $\til$ (with one caveat; see Remark~\ref{r:qmodmesodecomp}).  First, we construct a mesoprimary component for each $\til$-witness.  

\begin{defn}\label{d:qmodcoprincipal}
Fix a $Q$-module $T$.  A \emph{cogenerator} of $T$ is a non-nil element $t \in T$ with $q \cdot t \in N(T)$ for every nonunit $q \in Q$.  A $Q$-module $T$ is \emph{coprincipal} if it is $P$-mesoprimary and all its cogenerators lie in the same Green's class in $T_P$.  A congruence $\til$ on $T$ is \emph{coprincipal} if $T/\til$ is a coprincipal $Q$-module.  
\end{defn}

\begin{defn}\label{d:qmodcoprincomp}
Fix a $Q$-module $T$, a prime $P \subset Q$, and a witness $w \in T$ for $P$.  Let $\ol q$ denote the image of $q \in Q$ in $Q_P$, and $\ol t$ denote the image of $t \in T$ in $T_P$.  
\begin{itemize}
\item 
The \emph{order ideal $T_{\preceq w}^P$ cogenerated by $w$ at $P$} consists of those $a \in T$ whose image $\ol a \in T_P$ precedes $\ol w$ under Green's preorder.  

\item 
The \emph{congruence cogenerated by $w$ along $P$} is the equivalence relation $\til_w^P$ on $T$ that sets all elements outside of $T_{\preceq w}^P$ equivalent and sets $a \sim_w^P b$ whenever $\ol a$ and $\ol b$ differ by a unit in $T_P$ and $q \cdot \ol a = q \cdot \ol b = \ol w \in T_P$ for some $q \in Q_P$.  

\end{itemize}
\end{defn}

Lemma~\ref{l:qmodcoprincomp} justifies the nomenclature in Definition~\ref{d:qmodcoprincomp}.  

\begin{lemma}\label{l:qmodcoprincomp}
The congruence cogenerated by $w$ along $P$ is a coprincipal congruence on $T$ cogenerated by $w$.  Furthermore, $T/\til_w^P$ is properly connected, and if $T \setminus T_{\preceq w}^P$ is nonempty, then it is the nil class of $T/\til_w^P$.  
\end{lemma}

\begin{proof}
Let $T' = T/\til_w^P$.  Every non-nil element of $T'$ has the image of $w$ as a multiple, so $T'$ is properly connected, and it is clear that the image of $T \setminus T_{\preceq w}^P$ is nil modulo $\til_w^P$ as long as it is nonempty.  Furthermore, $w$ cogenerates $\til_w^P$ since the result of acting by any $p \in P$ lies outside $T_{\preceq w}^P$, and any $t \in T$ with non-nil image in $T'$ satisfies $q \cdot t = w$ for some $q \in Q$, so every cogenerator for $\til_w^P$ lies in the Green's class of $w$ in $T_P$.  

It remains to show that $T'$ is mesoprimary.  By Lemma~\ref{l:greensprimary}, $T_{\preceq w}^P$ has finitely many Green's classes in $T_P$, so each $p \in P$ acts nilpotently on $T'$ and thus $T'$ is $P$-primary.  Furthermore, for each $t \in T$ and for $a, b \in Q \setminus P$, we have $a \cdot t \sim_w^P b \cdot t$ if and only if $a \cdot w \sim_w^P b \cdot w$.  In particular, the $P$-prime congruences at the non-nil elements of $T'$ coincide, so by Theorem~\ref{t:qmodmesoprimary}, $T'$ is mesoprimary.  
\end{proof}

\begin{defn}\label{d:qmodmesodecomp}
Fix a $Q$-module $T$ and a congruence $\til$ on $T$.  
\begin{enumerate}
\item 
An expression $\til = \bigcap_i \til_i$ of $\til$ as the common refinement of finitely many mesoprimary congruences is a \emph{mesoprimary decomposition} if, for each component $\til_i$ with associated prime ideal $P \subset Q$, the $P$-prime congruences of $\til$ and $\til_i$ at each cogenerator for $\til_i$ coincide.  

\item 
A mesoprimary decomposition $\til = \bigcap_i \til_i$ is \emph{key} if, for each $P$-mesoprimary component $\til_i$, every cogenerator for $\til_i$ is a key $P$-witness for $\til$.  

\end{enumerate}
\end{defn}

We are now ready to give the main result of this paper.  Theorems~\ref{t:finitekey} and~\ref{t:qmodmesodecomp} together imply, as a special case, that every monoid module with at most one nil element admits a key mesoprimary decomposition (see Remark~\ref{r:qmodmesodecomp}).  

\begin{thm}\label{t:qmodmesodecomp}
Fix a congruence $\til$ on a $Q$-module $T$.  The common refinement of the coprincipal congruences cogenerated by the key witnesses of $\til$ identifies only the nil elements of $T/\til$.  
\end{thm}

\begin{proof}
The nil class of the congruence cogenerated by a witness $w \in T$ for $P$ contains the nil in the connected component of $w$ (if one exists), as well as every element outside of this connected component.  This means any $P$-coprincipal component identifies all of the nil elements of $T$.  

Now, fix distinct $a, b \in T$ and assume $a$ is not nil.  If $a$ and $b$ lie in distinct connected components, then any cogenerated congruence whose order ideal contains $a$ does not identify $a$ and $b$.  Assuming $a$ and $b$ lie in the same connected component, it suffices to find a monoid prime $P \subset Q$ and a key witness $w \in T$ for $P$ such that $a$ and $b$ are not equivalent under $\til_w^P$.  Fix a prime $P$ minimal among those containing the ideal $I = \{q \in Q : q \cdot a = q \cdot b\}$.  Notice that $I$ (and thus $P$) must be nonempty since $a$ and $b$ lie in the same connected~component.  

Since $P$ contains $I$, the elements $a$ and $b$ have distinct images $\ol a$ and $\ol b$ in $T_P$, and each $\ol q \in I_P$ also satisfies $\ol q \cdot \ol a = \ol q \cdot \ol b$.  By minimality of $P_P$ over $I_P$, there is a maximal Green's class among the elements $\{\ol q \in Q_P : \ol q \cdot \ol a \ne \ol q \cdot \ol b\}$.  Pick an element $q \in Q$ such that $\ol q$ lies in this Green's class, and set $w = q \cdot a \in T$.  Then $w$ is a key witness for $P$ by construction, and the localization of $\til_w^P$ does not equate $\ol a$ and $\ol b$ in $T_P$, so $\til_w^P$ does not equate $a$ and $b$ in $T$.  This completes the proof.  
\end{proof}

\begin{cor}\label{c:qmodmesodecomp}
Fix a $Q$-module $T$ and a congruence $\til$ on $T$.  If $T/\til$ has at most one nil element, then $\til$ admits a key mesoprimary decomposition.  
\end{cor}

\begin{proof}
Apply Theorem~\ref{t:qmodmesodecomp} and Lemma~\ref{l:qmodcoprincomp} to $T/\til$.  
\end{proof}

\begin{remark}\label{r:qmodmesodecomp}
Theorem~\ref{t:qmodmesodecomp} states that mesoprimary decomposition of monoid modules fails to distinguish nil elements from one another, and that this is the only obstruction to constructing mesoprimary decompositions in this setting.  Fortunately, for the purposes of decomposing graded modules over a monoid algebra, these elements all correspond to zero in the module and thus are indistinguishable.  
\end{remark}

\section{The category of tightly graded $\kk[Q]$-modules}%
\label{s:tightlygraded}

Section~\ref{s:qmod} defined the category \QMod~of $Q$-modules, the setting in which mesoprimary decomposition of monoid congruences is generalized in the prior sections of this paper.  In this section, we define the category $\cB_Q$ of tightly graded $\kk[Q]$-modules (Definition~\ref{d:bmod}), the objects of which are graded by objects of \QMod.  It is to these graded modules that we generalize mesoprimary decomposition of binomial ideals in the subsequent sections of this paper.  

\begin{defn}\label{d:graded}
Fix a $Q$-module $T$ and a $\kk[Q]$-module $M$.  
\begin{itemize}
\item 
$M$ is \emph{graded by $T$} (or just \emph{$T$-graded}) if there exist a collection of finite dimensional vector spaces $\{M_t\}_{t \in T}$ such that $M \iso \bigoplus_{t \in T} M_t$ as Abelian groups, and for each $q \in Q$, ${\bf t}^q \cdot M_t \subset M_{q \cdot t}$.  

\item 
The grading of $M$ by $T$ is \emph{fine} (or $M$ is \emph{finely-graded by $T$}) if $\dim_\kk M_t \le 1$ for each $t \in T$.  

\item 
A fine grading of $M$ by $T$ is \emph{tight} (or $M$ is \emph{tightly-graded by $T$}) if
\begin{itemize}
\item $M_t \ne 0$ for each non-nil $t \in T$, 
\item the orbit of each $\infty \in N(T)$ with $M_\infty = 0$ is properly connected, and
\item whenever $m \in M_t$ is nonzero with ${\bf x}^q \cdot m = 0$, we have $\dim_\kk M_{q \cdot t} = 0$.  
\end{itemize}

\end{itemize}
\end{defn}

\begin{remark}\label{r:tight}
A tight grading of a $\kk[Q]$-module $M$ by a $Q$-module $T$ ensures that we can determine enough of the structure of $M$ from the grading.  The first condition ensures that $T$ does not have any unnecessary elements, and the second ensures each connected component has its own nil (see~Proposition~\ref{p:bmod}).  Example~\ref{e:tightcond} demonstrates what can cause the third condition to fail.  
\end{remark}

\begin{example}\label{e:tightcond}
The $\kk[x]$-module $M = \<x^2\> \oplus (\kk[x]/\<x^2\>)$ is finely graded by $\NN$, but since $x \cdot (0,x)$ is zero, this grading is not tight.  However, $M$ is tightly graded by the disjoint union of $\<2\> \subset \NN$ (which tightly grades $\<x^2\> \subset \kk[x]$) and $\NN/\<2\>$ (which tightly grades $\kk[x]/\<x^2\>$).  This grading more accurately reflects the algebraic structure of~$M$.  
\end{example}

In order to study finely graded $\kk[Q]$-modules, it suffices to consider tight gradings.  In particular, every tight grading is fine, and Theorem~\ref{t:tight} shows that a tight grading can be recovered from any fine grading by chosing an appropriate $Q$-module.  

\begin{thm}\label{t:tight}
Fix a $\kk[Q]$-module $M$ finely graded by a $Q$-module $T$.  Then there exists a $Q$-module that tightly grades $M$.  
\end{thm}

\begin{proof}
We construct the desired $Q$-module in two steps.  First, define a $Q$-module $T'$ that, as a set, consists of those $t \in T$ for which $\dim_\kk M_t = 1$, along with a distinguished element $\infty$.  Given $t \in T'$ and $q \in Q$, define $q \cdot t \in T'$ by
\begin{center}
$q \cdot t = \left\{\begin{array}{ll}
q \cdot t \in T & {\bf x}^q M_t \ne 0 \\
\infty & \text{otherwise}
\end{array}\right.,$
\end{center}
that is, the result of acting on $t$ by $q$ in $T$ if ${\bf x}^q M_t \ne 0$, and $\infty \in T'$ otherwise.  The $Q$-module $T'$ also finely grades $M$ since each nonzero $M_t$ for $t \in T$ has a corresponding degree in $T'$.  Moreover, the only degree $t \in T'$ with $\dim_\kk M_t = 0$ is $t = \infty$, and whenever ${\bf x}^q M_t = 0$, we  have $q \cdot t = \infty$.  In particular, $T'$ satisfies the first and third conditions for a tight grading in Definition~\ref{d:graded}.  

Next, let $T_1', \ldots, T_r'$ denote the distinct $Q$-orbits of $T' \setminus \{\infty\}$, and let $T''$ denote the the disjiont union of the sets $T_1', \ldots, T_r'$ together with distinguished elements $\infty_1, \ldots, \infty_r$.  Define a $Q$-module structure on $T''$ so that $\infty_i$ is nil for each $i \le r$, and 
\begin{center}
$q \cdot t = \left\{\begin{array}{ll}
\infty_i & q \cdot t = \infty \\
q \cdot t \in T_i' & \text{otherwise}
\end{array}\right.$
\end{center}
for $t \in T_i'$ and $q \in Q$.  Since the orbit of each nil $\infty_i$ of $T''$ with trivial support in $M$ is properly connected, $T''$ tightly grades $M$.  This completes the proof.  
\end{proof}

\begin{defn}\label{d:lhom}
Suppose $M$ and $N$ are $\kk[Q]$-modules, graded by $Q$-modules $T$ and $U$, respectively.  
\begin{itemize}
\item 
A homomorphism $\phi:M \to N$ is said to be \emph{graded with degree} $\psi \in \Hom_Q(T,U)$ if for each $t \in T$, we have $\phi(M_t) \subset N_{\psi(t)}$.  

\item 
Let $\underline{\Hom}_R(M,N)_\psi$ denote the set of morphisms $\phi:M \to N$ of degree $\psi$, and write 
$$\underline{\Hom}_R(M,N) = \bigoplus_{\psi \in \Hom_Q(T,U)} \underline{\Hom}_R(M,N)_\psi$$
for the set of \emph{graded homomorphisms} from $M$ to $N$.  $\underline{\Hom}_R(M,N)$ is naturally a $\Hom_Q(T,U)$-graded $\kk[Q]$-module with the action of ${\bf t}^q$ given by $({\bf t}^q \cdot \phi)(m) = \phi({\bf t}^q \cdot  m)$ for each $m \in M$, $q \in Q$.  

\item 
A homomorphism $\phi \in \underline{\Hom}_R(M,N)$ is \emph{homogeneous} if it is a sum of homomorphisms with homogeneous degree in $\Hom_Q(T,U)$.  

\end{itemize}
\end{defn}

\begin{defn}\label{d:bmod}
The \emph{category of tightly graded $\kk[Q]$-modules} is the category $\cB_Q$ whose objects are pairs $(M,T)$ consisting of a $Q$-module $T$ together with a $\kk[Q]$-module $M$ tightly graded by $T$, and whose morphisms are graded $\kk[Q]$-module homomorphisms.  When there is no confusion, we often write $M \in \cB_Q$ to denote the $\kk[Q]$ module and use $T_M$ to denote the $Q$-module which tightly grades $M$.  
\end{defn}

\begin{prop}\label{p:bmod}
The category $\cB_Q$ is closed under taking direct sums and tensor products.  More precisely, given two $\kk[Q]$-modules $M$ and $N$ tightly graded by $Q$-modules $T$ and $U$, respectively, the direct sum $M \oplus N$ is naturally graded by $T \oplus U$, and the tensor product $M \otimes_{\kk[Q]} N$ is naturally graded by $T \otimes_Q U$.  
\end{prop}

\begin{proof}
The homogeneous elements of $M \oplus N$ have the form $(m,0), (0,n)$ for homogeneous $m \in M_t$, $n \in N_u$, and the degree map is given by $\deg(m,0) = t \in T \oplus U$, $\deg(0,n) = u \in T \oplus U$.  The homogeneous elements of $M \otimes_{\kk[Q]} N$ have the form $m \otimes n$ for homogeneous $m \in M_t$, $n \in N_u$, and the degree map is given by $\deg(m \otimes n) = t \otimes u$.  Notice that
$$\deg(m \otimes ({\bf t}^q \cdot n)) = t \otimes (q \cdot u) = (q \cdot t) \otimes u = \deg(({\bf t}^q \cdot m) \otimes n)$$
so this degree map is well defined.  
\end{proof}

\section{Mesoprimary $\kk[Q]$-modules}%
\label{s:binmodmesoprimary}

In this section, we define binomial submodules of tightly graded $\kk[Q]$-modules (Definition~\ref{d:binomialsubmod}), generalizing the concept of ``binomial ideal''.  We define mesoprimary binomial submodules (Definition~\ref{d:binmesoprimary}), which, like mesoprimary binomial ideals, are characterized by their unique associated mesoprime ideal (Theorem~\ref{t:binmesoprimary}).  

\begin{defn}\label{d:binomialsubmod}
Fix a tightly $T$-graded $\kk[Q]$-module $M$ and a nonzero element $m \in M$.  
\begin{itemize}
\item 
The element $m$ is a \emph{monomial} if it is homogeneous under the $T$-grading.  

\item 
The element $m$ is a \emph{binomial} if it is a sum of at most two monomial elements.  

\item 
A submodule $N \subset M$ is \emph{monomial} (resp. \emph{binomial}) if it is generated by monomial (resp. binomial) elements.  

\end{itemize}
\end{defn}

\begin{lemma}\label{l:binmodcong}
Fix a tightly $T$-graded $\kk[Q]$-module $M$, and a binomial submodule $N \subset M$.  Let $\til_N$ denote the equivalence relation on $T$ which sets $a \sim_N b$ whenever $m_a + m_b \in N$ for some nonzero $m_a \in M_a, m_b \in M_b$.  Then $\til_N$ is a congruence on $T$, and $\ol M = M/N$ is tightly graded by $\ol T = T/\til_N$.  
\end{lemma}

\begin{proof}
It is clear that $\til_N$ is a congruence on $T$, and that $\ol T$ finely grades $\ol M$.  If $\ol t \in \ol T$ is non-nil, then each representative $t \in T$ for $\ol t$ is non-nil, meaning $\dim_\kk \ol M_\ol t = 1$.  Additionally, if ${\bf x}^q \cdot \ol m = 0$ for some nonzero $\ol m \in \ol M_\ol t$, then any nonzero $m \in M$ whose image in $\ol M$ equals $\ol m$ satisfies ${\bf x}^q \cdot m = 0$.  Since $T$ tightly grades $M$, this means $\dim_\kk M_{q \cdot t} = 0$, so $\dim_\kk \ol M_{q \cdot \ol t} = 0$.  Lastly, if $\ol t \in \ol T$ is nil and $\dim_\kk \ol M_\ol t = 0$, then each representative $t \in T$ of $\ol t$ is nil and satsfies $\dim_\kk M_t = 0$.  As such, $N$ cannot contain any nonzero binomials whose monomials have image of degree $\ol t$ in $\ol M$, so $\ol T$ tightly grades $\ol M$, as desired.  
\end{proof}

\begin{defn}\label{d:binmesoprimary}
A tightly $T$-graded $\kk[Q]$-module $M$ is \emph{mesoprimary} if $T$ is a mesoprimary $Q$-module and $M_\infty = 0$ for each nil $\infty \in T$.  A binomial submodule $N \subset M$ is \emph{mesoprimary} if $M/N$ is a mesoprimary $\kk[Q]$-module.  
\end{defn}

\begin{example}\label{e:mesoprimaryideal}
If $I \subset \kk[Q]$ is a binomial ideal, then $\kk[Q]/I$ is tightly $T$-graded for $T = Q/\til_I$.  Moreover, $\kk[Q]/I$ is mesoprimary when $T$ is mesoprimary and $I$ is maximal among binomial ideals inducing the congruence $\til_I$.  By Lemma~\ref{l:primarymonoidcase}, this is precisely when $I$ is mesoprimary; see \cite[Definition~10.4]{kmmeso}.  
\end{example}



Definition~\ref{d:assocmesoprime} generalizes \cite[Definition~12.1]{kmmeso}.  

\begin{defn}\label{d:assocmesoprime}
Fix a tightly $T$-graded $\kk[Q]$-module $M$, and a binomial submodule $N \subset M$.  Fix a monoid prime $P \subset Q$, and let $G$ denote the unit group of $Q_P$.  

\begin{itemize}

\item 
The \emph{monomial localization of $M$ at $P$}, denoted $M_P$, is the $\kk[Q]_P$-module obtained by adjoining to $M$ the inverses of all monomials outside of the monomial ideal $\mm_P = \<{\bf x}^p : p \in P\>$.  

\item 
An element $w \in T$ is an \emph{$N$-witness} for $P$ if $w$ is a $\til_N$-witness for $P$ on $T$, and $w$ is \emph{essential} if a nonzero element of $M_w$ is minimal  (under Green's preorder) among the monomials of some element $m \in M$ annihilated by $\mm_P$ in $M_P/N_P$.  A nonzero monomial $m_w \in M_w$ is called a \emph{monomial $P$-witness} for $N$.  

\item
Fix a monomial $N$-witness $m \in M_w$ for $P$.  The \emph{stabalizer} of $w$ along a prime $P \subset Q$ is the subgroup $K_w^P \subset G_P$ fixing the class of $w$ in $T_P$.  The \emph{character at $m$} is the homomorphism $\rho:K_w^P \to \kk^*$ such that $(I_{\rho,P})_P = \ann(\ol m) + \mm_P$, where $\ol m$ denotes the image of $m$ in $M_P$.  The \emph{$P$-mesoprime of $M$ at $m$} is the mesoprime ideal $I_{\rho,P}$.  

\item 
If $m \in M_w$ is an essential $N$-witness for $P$, we say the mesoprime $I_{\rho,P}$ is \emph{associated to $N$}, and $m$ is an \emph{$N$-witness for $I_{\rho,P}$}.  

\end{itemize}
\end{defn}

Proposition~\ref{p:finiteessential} generalizes \cite[Lemma~12.4]{kmmeso}, and can be proven using a similar argument to Theorem~\ref{t:finitekey}.  

\begin{prop}\label{p:finiteessential}
Any binomial submodule of a tightly graded $\kk[Q]$-module has finitely many essential witnesses.  
\end{prop}

Theorem~\ref{t:binmesoprimary} generalizes \cite[Proposition~12.10]{kmmeso}, and may fail to hold if the grading $Q$-module is not properly connected; see Example~\ref{e:binmesoprimarypropcon}.  

\begin{thm}\label{t:binmesoprimary}
Fix a tightly $T$-graded $\kk[Q]$-module $M$ with $T$ properly connected, and fix a binomial submodule $N \subset M$.  Then $N$ is mesoprimary if and only if $N$ has exactly one associated mesoprime.  
\end{thm}

\begin{proof}
Let $\ol M = M/N$ and $\ol T = T/\til_N$.  First, suppose $N$ has exactly one associated mesoprime $I_{\rho,P}$.  This means the $P$-prime congruences agree at all $N$-witnesses, so $\ol T$ has exactly one associated prime congruence, and thus is mesoprimary.  Furthermore, if $\ol M_\infty \ne 0$ for some nil $\infty \in \ol T$, then the associated mesoprime at any nonzero element of $\ol M_\infty$ differs from the associated mesoprime in any non-nil degree $t \in T$ with $q \cdot \ol t = \infty$ for some $q \in Q$.  

Next, suppose $N$ is mesoprimary.  Fix $N$-witnesses $w, w' \in T$ with associated mesoprimes $I_{\rho,P}$ and $I_{\rho',P}$, respectively, and suppose $w = q \cdot w'$.  Since $T$ has exactly one associated prime congruence, multiplication by ${\bf t}^q$ induces an isomorphism $I_{\rho,P} \to I_{\rho',P}$, that is, the associated mesoprimes at $w$ and $w'$ coincide.  Since $T$ is properly connected, this shows that the associated mesoprimes at every $M$-witness coincide.  
\end{proof}

\begin{example}\label{e:binmesoprimarypropcon}
Resuming notation from Theorem~\ref{t:binmesoprimary}, the result may fail to hold in general if $T$ is not properly connected.  Let $I = \<x - 1, y\>, J = \<x - 2, y\> \subset \kk[x,y]$, $M = \kk[x,y]/I \oplus \kk[x,y]/J$, and $T = Q/\til_I \oplus Q/\til_J$ with nils identified.  Even though $M$ is mesoprimary, it has two distinct associated mesoprimes, one for each connected component of $Q/\til_I \oplus Q/\til_J$.  
\end{example}

\section{Mesoprimary decomposition of binomial submodules}%
\label{s:binmesodecomp}

In this section, we use mesoprimary submodules (Definition~\ref{d:binmesoprimary}) to construct a mesoprimary decomposition of any binomial submodule of a tightly graded $\kk[Q]$-module (Theorem~\ref{t:binmesodecomp}), thus completing our answer to Problem~\ref{pr:kmbinmod}.  

\begin{defn}\label{d:bincogen}
Fix a tightly $T$-graded $\kk[Q]$-module $M$, a binomial submodule $N \subset M$, a prime $P \subset Q$, and a monomial $N$-witness $m_w \in M_w$ for $P$.  

\begin{itemize}

\item 
The \emph{monomial submodule cogenerated by $w$ along $P$} is the submodule $M_w^P(N) \subset M$ generated in those degrees $u \in T$ that lie outside of the order ideal $T_{\le w}^P$ cogenerated by $w$ along $P$ under the congruence $\til_N$.  

\item 
The \emph{$P$-mesoprime component of $N$ cogenerated by $w$} is the preimage $W_w^P(I)$ in $M$ of the submodule $N_P + C_w^P(N) + M_w^P(N) \subset M_P$, where 
$$C_w^P(N) = \<m_a - m_b : {\bf x}^q (m_a - m_b) \in M_w\>.$$

\end{itemize}
\end{defn}

\begin{example}\label{e:bincogen}
Resume notation from Definition~\ref{d:bincogen}.  If $M = \kk[Q]$ and $N = I$ is a binomial ideal, then $C_w^P(I) = I_{\rho,P} \subset M$, so Definition~\ref{d:bincogen} is equivalent to \cite[Definition~12.13]{kmmeso} in this case.  In general, we have $I_{\rho,P}M_P \subset C_w^P(I)$, but equality need not hold.  Let $Q = \NN^2$, $M = (\kk[x,y]/\<y - xy, y^2\>)^{\oplus 2}$, and $T$ the $Q$-module that tightly grades $M$.  Write $e_1$ and $e_2$ for the generators of the summands of $M$, and let $N = \<ye_1 - ye_2\>$, $w = ye_1$.  Then the associated mesoprime at $w$ is $I_{\rho,P} = \<1 - x, y\>$ and $W_w^P(N) = \<e_1 - e_2, e_1(1 - x), e_2(1 - x)\>$.  Here, $W_w^P(N)$ must contain the binomial $e_1 - e_2$ in order to induce the desired coprincipal congruence on $T$, and this is not captured in the combinatorial data of $I_{\rho,P}$ alone.  
\end{example}

\begin{prop}\label{p:bincogen}
Fix a tightly $T$-graded $\kk[Q]$-module $M$, a binomial submodule $N \subset M$, a prime $P \subset Q$, and an $N$-witness $w \in T$ for $P$.  The submodule $W_w^P(N)$ is mesoprimary with associated mesoprime $I_{\rho,P}$.  In particular, if $N$ induces the congruence $\til$ on $T$, then $W_w^P(N)$ induces the coprincipal congruence $\til_w^P$.  
\end{prop}

\begin{proof}
Let $\app$ denote the congruence induced by $W_w^P(N)$.  
We can see from the definitions that $\til_w^P$ refines $\app$, so it remains to show that no further relations are added.  Since the congruences induced by $N_P$ and $C_w^P(N)$ both refine $\til_w^P$, it suffices to show that the nil class of $\app$ is identical to that of $\til_w^P$.  That is, we must check that whenever $a \sim b$ and $a \approx b$ for non-nil $a, b \in T$, these relations are induced by same binomial elements in $N_P$ and $C_w^P(N)$.  Suppose $m_a - m_b \in N$ for nonzero $m_a \in M_a$, $m_b \in M_b$ such that $a \approx b$ but $m_a, m_b \notin M_w^P(N)$.  Since $a, b \in T_{\preceq w}^P(\til)$, we can find $q \in Q$ so that $q \cdot a$ and $q \cdot b$ are Green's equivalent to $w$ in $T_P$.  This means ${\bf x}^q m_a - {\bf x}^q m_b \in M_w \cap N$, and since $w$ is not in the nil class of $\til$, we must have ${\bf x}^q m_a - {\bf x}^q m_b = 0$.  In particular, this means $m_a - m_b \in C_w^P(N)$, as desired.  
\end{proof}

\begin{thm}\label{t:binmesodecomp}
Any binomial submodule $N$ of a tightly $T$-graded $\kk[Q]$-module $M$ equals the intersection of the coprincipal components cogenerated by its essential witnesses.  
\end{thm}

\begin{proof}
Pick an element $m \in M$ outside of $N$.  The goal is to find an essential witness $w$ and a monoid prime $P$ such that $m$ lies outside of the coprincipal component $W_w^P(N)$.  

First, suppose the image of $m$ lies outside of the localization $N_P$ along a maximal prime $P$ of $Q$.  Replacing $m$ with a monomial multiple of $m$, it suffices to assume that $\mm_P m \subset N$, that is to say, $m$ is annihilated (modulo $N$) by the maximal ideal $\mm_P$.  This means some monomial of $m$ has as its graded degree an essential witness $w$ for $P$.  By the minimality of $w$, the image of $m$ modulo $W_w^P(N)$ lies in the image of $M_w$ modulo $W_w^P(N)$, which is nonzero by Proposition~\ref{p:bincogen}.  

Next, suppose the image of $m$ under localization along some non-maximal monoid prime $P$ lies outside of $N_P$.  The above argument implies that the localized image of $m$ lies outside of some $P$-coprincipal component of $N_P$, which by Definition~\ref{d:bincogen} equals the localization $W_P$ of a $P$-coprincipal component $W$ of $N$.  Since localizing $W$ along $P$ is injective, this completes the proof.  
%
%
\end{proof}

\section{Primary decomposition of binomial submodules}%
\label{s:binprimarydecomp}

In this section, we extend the results of \cite[Section~15]{kmmeso} to construct a primary decomposition for any binomial submodule of a tightly graded $\kk[Q]$-module.  More specifically, the results presented in this section directly parallel those found in \cite[Proposition~15.1]{kmmeso}, \cite[Corollary~15.2]{kmmeso}, and \cite[Theorem~15.4]{kmmeso}, used to construct primary decompositions of mesoprimary binomial ideals over an algebraically closed field.  Corollary~\ref{c:primarydecomposition}, together with Theorem~\ref{t:binmesodecomp}, yield a combinatorial method of primarily decomposing a binomial submodule whenever $\kk = \ol\kk$ is algebraically closed.  

\begin{prop}\label{p:assocprimes}
Fix a tightly $T$-graded $\kk[Q]$-module $M$ and a mesoprimary binomial submodule $N \subset M$.  The associated primes of $N$ are precisely the associated primes of its unique associated mesoprime $I_{\rho,P}$.  In particular, $N$ is primary if and only if its associated mesoprime is prime.  
\end{prop}

\begin{proof}
Suppose $\ol T_P = T_P/\til_N$ has unit group $G$.  Notice that localizing along $P$ induces an injection $M/N \hookrightarrow (M/N)_P$ since the monomials outside of $\mm_P$ are nonzerodivisors on the quotient modulo any $P$-mesoprimary ideal.  Moreover, by Theorems~\ref{t:binmesoprimary} and~\ref{t:qmodmesoprimary}, $(M/N)_P$ has finitely many nonzero $(\ol T_P/G)$-graded pieces, all isomorphic to $(\kk[Q]/I_{\rho,P})_P$.  The partial order on $\ol T_P/G$ afforded by Lemma~\ref{l:greensprimary} induces a filtration of $(M/N)_P$ by $M_P$-submodules, each free of finite rank as a module over $(\kk[Q]/I_{\rho,P})_P$.  
\end{proof}

\begin{thm}\label{t:primarydecomposition}
Suppose $\kk = \ol\kk$ is algebraically closed.  Fix a tightly $T$-graded $\kk[Q]$-module $M$ and a mesoprimary binomial submodule $N \subset M$.  If $I_{\rho,P}$ is the unique associated mesoprime of $N$ and $I_{\rho,P} = \bigcap_\sigma I_{\sigma,P}$ is the unique primary decomposition of~$I_{\rho,P}$ from \cite[Proposition~11.9]{kmmeso}, then 
$$N = \textstyle\bigcap_\sigma (N + I_{\sigma,P}M)$$
is the unique minimal primary decomposition of $N$.  
\end{thm}

\begin{proof}
Each submodule $N + I_{\sigma,P}M \subset M$ is binomial and mesoprimary, and thus primary by Proposition~\ref{p:assocprimes}.  The intersection $\textstyle\bigcap_\sigma (N + I_{\sigma,P}M)$ certainly contains $N$, and the converse follows from the equality $N = N + I_{\rho,P}M$.  
\end{proof}

\begin{cor}\label{c:primarydecomposition}
Suppose $\kk = \ol\kk$ is algebraically closed.  Every binomial submodule $N \subset M$ of a tightly $T$-graded $\kk[Q]$-module $M$ admits a primary decomposition in which each component is again binomial.  
\end{cor}

\begin{proof}
Apply Theorem~\ref{t:binmesodecomp} to construct a mesoprimary decomposition for $N$, then apply Theorem~\ref{t:primarydecomposition} to each mesoprimary component.  
\end{proof}

\begin{remark}\label{r:soccular}
In \cite{kmo}, mesoprimary decomposition is used to combinatorially construct irreducible decompositions of binomial ideals, using ``soccular decomposition'' as an intermediate step.  It remains an interesting question to extend soccular decomposition to tightly graded modules; we record this here.  
\end{remark}

\begin{prob}\label{pr:soccular}
Extend soccular decomposition \cite{kmo} to tightly graded modules.  
\end{prob}

\section{Acknowledgements}%
\label{s:acknowledgements}%

The author would like to thank Ezra Miller for his energy and patience during the author's time as a graduate student, and for motivating the work presented here.  The author would also like to thank Thomas Kahle and Laura Matusevich for their continued encouragement and for many useful discussions.  Much of this work was completed while the author was a graduate student, funded in part by Ezra Miller's NSF Grant DMS-1001437.  Some results from Sections~\ref{s:qmod}-\ref{s:qmodmesodecomp} also appear in~\cite{mesothesis}.


\end{document}